\documentclass{amsart}
\usepackage{amsfonts}
\usepackage{graphicx}
\usepackage{amscd}
\usepackage{amsmath}
\usepackage{amssymb}
\usepackage{hyperref}
\usepackage{amsthm}
\usepackage[latin1]{inputenc}
\theoremstyle{plain}

\newtheorem{corollary}{\bf Corollary}[section]

\newtheorem{lemma}{\bf Lemma}

\newtheorem{proposition}{\bf Proposition}[section]
\newtheorem{remark}{Remark}[section]

\newtheorem{theorem}{\bf Theorem}[section]
\newtheorem*{theorem*}{Theorem}
\newtheorem{conj}{Conjecture}
\newtheorem{qu}{Question}
\newtheorem*{qu*}{Question}
\numberwithin{equation}{section}

\newcommand\dM{\mathrm{dM}}

\newcommand\dm{\mathrm{dm}}
\newcommand\h{\mathcal{H}}

\usepackage{color}

\begin{document}

\title[]{On the spectra of geometric operators evolving with geometric flows}
\author{R. R. Mesquita$^1$ $\&$ D. M. Tsonev$^2$}
\address{$^{1,2}$Departamento de Matem\'atica-UFAM, 69077-070-Manaus-AM-BR}
\email{$^{1}$mesquitaraul532@gmail.com}
\email{$^{2}$dragomir\_tsonev@yahoo.co.uk}
\keywords{Witten-Laplacian, Eigenvalues, Monotonicity of eigenvalues, Ricci flow,  Ricci-Bourguignon flow, Yamabe flow, Bochner formula, Reilly formula}
\urladdr{$^{1,2}$http://www.mat.ufam.edu.br}

\begin{abstract} 
In this work we generalise various recent results on the evolution and monotonicity of the eigenvalues of certain geometric operators under specified geometric flows. Given a closed, compact Riemannian manifold $\big(M^n,g(t)\big)$ and a smooth function $\eta\in C^{\infty}(M)$ we consider the family of operators $\mathbb{L}=\Delta - g(\nabla\eta,\nabla\cdot)+cR$, where $R$ is the scalar curvature and $c$ is some real constant. We define a geometric flow on $M$ which encompasses the Ricci, the Ricci - Bourguignon and the Yamabe flows. Supposing that the metric $g(t)$ evolves along this general geometric flow we derive a formula for the evolution of the eigenvalues of $-\mathbb{L}$ and prove monotonicity results for the eigenvalues of both $-\Delta + g(\nabla\eta,\nabla\cdot)$ and $-\mathbb{L}$. We then prove Reilly-type formula for the operator $\mathbb{L}$ and employ it to establish an upper bound for the first variation of the eigenvalues of $-\mathbb{L}$. Finally, in the pursuit of a theoretical explanation of our generalisations, we formulate two conjectures on the monotonicity of the eigenvalues of  Schr\"{o}dinger operators.
\end{abstract}
\maketitle


\section{Introduction}
It is a natural question to ask how do the eigenvalues of given geometric operator evolve along a geometric flow. Various results in this direction have been obtained in the past ten years. More concretely, Di Cerbo \cite{Fab} studied the aforementioned question in the case of the Laplace-Beltrami operator $\Delta$ evolving along the Ricci flow. At about the same time, Cao \cite{cao1}
first studied the eigenvalues of the operator $-\Delta + \frac{1}{2}R$ under the Ricci flow with the scalar curvature assumption $R\geq0$, and soon after  he extended his techniques  to the eigenvalues of $-\Delta + cR$, $c\geq \dfrac{1}{4}$, evolving with the Ricci flow but this time without any curvature assumptions \cite{cao2}. In 2012, Cao, Hou and Ling \cite{cao3} worked out the case for the operator $-\Delta + aR$, $0<a\leq\dfrac{1}{2}$, also under the Ricci flow. The evolution and monotonicity of the latter operator have also been studied under the Ricci-Bourguignon flow, provided $a\neq0$ \cite{Zeng}. Very recently, the evolution and the monotonicity of the eigenvalues of the operator $\Delta_{\eta}+cR$ have been studied both under the Ricci flow \cite{fang1, fang2} and the Yamabe flow \cite{fang3}. The operator $\Delta_{\eta}$ is known in the literature  as the {\it Witten - Laplacian}, or the {\it drifting Laplacian}. Two remarks should be brought forward at this juncture. Firstly, for notational convenience, and in contrast to the more common symbol $\Delta_{\eta}$, we shall prefer to write $L$ for the Witten-Laplacian instead. Secondly, and more importantly, it should be noticed that, to the best of our knowledge, most of the evolution formulas known in the literature have been derived by making use of the so-called $\mathcal{F}$-{\it entropy functional}. Introduced by Perelman \cite{perelman}, it is known to be nondecreasing under the Ricci flow coupled to a backward heat-type equation which implies the monotonicity of the first eigenvalue of $-4\Delta+R$ along the Ricci flow. Notwithstanding, our approach is not relying upon the Perelman's $\mathcal{F}$-entropy functional, and our method of proof will be explained in detail in the next section.

The goal of this work is twofold. On the one hand, it is  to generalise the results on the evolution and monotonicity known so far without using the $\mathcal{F}$-entropy functional. On the other hand, it is to attempt a theoretical explanation of this, rather natural, generalisation. Motivated by the works cited above, we shall consider closed $n$-dimensional Riemannian manifold $(M, g)$  and a family of geometric operators acting on the smooth functions on the manifold  defined by 
\begin{equation*}
\mathbb{L}=\Delta - g(\nabla\eta,\nabla\cdot)+cR.
\end{equation*}
Here $\eta\in C^{\infty}(M)$, $\Delta$ is the standard Laplace-Beltrami operator on $M$, $R$ is the scalar curvature of the metric $g$ and $c$ is some real constant. Clearly, $\mathbb{L}$ can be thought of as a generalisation of $L$ as much as it can be thought of as a generalisation of either $\Delta$ or $\Delta+cR$. It is also clear that if the metric $g(t)$ varies with time, so will $\mathbb{L}$. 
Now, our idea is to introduce a ``flow'' which encompasses the Ricci, the Ricci-Bourguignon and the Yamabe flows, and to study the evolution and the monotonicity of the eigenvalues of the operator $-\mathbb{L}$ along this flow.  We shall then wish to see how our evolution formula compares to the already known cases. For this reason, we shall suppose that the metric $g(t)$ on our closed Riemannian manifold evolves with the following PDE
\begin{eqnarray}\label{RBY}
\frac{\partial g_{t}}{\partial t}=-2aRic_{t}+2\left(\rho R_{t}-\frac{\varphi}{n}r_{t}\right)g_{t},
\end{eqnarray}
where $a,\rho$ are real constants, $\varphi=-a+n\rho$, $r_{t}=\dfrac{\int_{M}R(t)\dm_{t}}{\int_{M}\dm_{t}}$ is the average scalar curvature, and the subscript $t$ means evolution in time. Evidently, this latter formula nicely generalises the three most studied flows in differential geometry. It is indeed immediately observed that for $a=1$ and $r_{t}=0$ the flow \eqref{RBY} reduces to the {\it Ricci-Bourguignon flow}. The {\it normalised Ricci flow} arises if $a=1$ and $\rho=0$, and its {\it unnormalised} version if, in addition, $r_{t}=0$. For $a=0$ and $\rho=-\frac{1}{2}$ we have the {\it normalised Yamabe flow}. 

The careful reader should immediately notice that  \eqref{RBY} does not always behave nicely from the PDE viewpoint. Indeed, the short-time existence as well as the uniqueness of the solution of the equation \eqref{RBY} for arbitrary initial data are not guaranteed for all values of the parameters $a,\,\rho$ and $\varphi$. An important and natural question is for which values of the latter parameters the evolution equation \eqref{RBY} admits unique short-time solutions. We shall, nevertheless, put this question aside as it is not of crucial importance for our purposes. It will suffice to consider only those values of $a,\,\rho$ and $\varphi$ for which the evolution equation \eqref{RBY} admits unique solutions. At the very least, as long as the parameters are chosen so that \eqref{RBY} restricts to either Ricci, Ricci-Bourguignon or Yamabe flow we shall be on the safe side. 

Following the idea just described we shall derive in Section \ref{evol} an evolution formula for the eigenvalues of the operator $-\mathbb{L}$ evolving with the flow \eqref{RBY}. Remarkably, our formula  nicely generalises  all of the evolution formulas appearing in the articles cited in the bibliography. We shall also prove, under some further hypotheses, that the eigenvalues of both $-L$ and  $-\mathbb{L}$ are non-decreasing. Thus, combining our results with those already known in the literature, the following question should be quite natural to ask.  
\begin{qu}\label{qu1}
To what extent the behaviour of the eigenvalues of the operators $\Delta$, $\Delta+cR$ and $L$ can be determined by the behaviour of the eigenvalues of $\mathbb{L}$? Or rather, are the monotonicity of the eigenvalues of the latter three operators somewhat inherited from the monotonicity of the eigenvalues of the Laplacian $\Delta$? 
\end{qu}

The paper is organised as follows. In Section \ref{prem} we fix notation, comment upon our method of proof and list without proofs all the main formulae which are to be exploited in the course of this work. In Section \ref{evol} we derive the evolution formula for the eigenvalues of $-\mathbb{L}$ under the flow \eqref{RBY}. In Section \ref{mono} we first study the monotonicity of the eigenvalues of $-L$ and then the monotonicity of the eigenvalues of the more general operator $-\mathbb{L}$. In Section \ref{appl} we derive the Reilly-type formula for the operator $\mathbb{L}$ which, combined with the evolution formula from Section \ref{evol} and the assumption that the scalar curvature $R$ is constant, results in an upper bound for the first variation of  eigenvalue of $-\mathbb{L}$.  In Section \ref{homotopy} we provide a conjectural answer of  Question \ref{qu1}. We also make some concluding remarks in Section \ref{concl}. We conclude this paper by appending an alternative proof  of our evolution formula as well as providing the proof of the monotonicity of the eigenvalues of $-\mathbb{L}$. 
\newline\\
{\bf Acknowledgements} Jose N. V. Gomes and Marcus Marrocos and others...


\section{Preliminaries}\label{prem}

\subsection{Method of proof.}  
We shall consider throughout an $n$-dimensional oriented compact Riemannian manifold $(M,g)$ without boundary endowed with a {\it weighted measure} of the form $\dm=e^{-\eta}\dM$. Here $\dM$ is the standard volume element and the smooth function $\eta:M\rightarrow\mathbb{R}$ is often referred to as the {\it drifting function}. The evolution equation for the metric $g(t)$ will be either given by \eqref{RBY} or some appropriate modification of it. As already mentioned in the introduction, we shall always assume that the flows we  work with admit unique short-time solutions for the metric $g(t)$. Given the family of operators $\mathbb{L}=\Delta - g(\nabla\eta,\nabla\cdot)+cR$ on $M$ we shall consider the eigenvalue problem
$$-\mathbb{L}u(x,t)=\lambda(t)u(x,t).$$
It should be noticed that in the sequel $\lambda(t)$ is to be understood an eigenvalue of the operator 
$-\mathbb{L}$ unless otherwise stated.

Many authors follow Perelman's idea that the Ricci flow can be thought of as a gradient flow. It is  a beautiful idea with various applications. However, our work  is not relying upon it and we shall therefore skip the discussion on the $\mathcal F$-functional and its implications. The interested reader unfamiliar with this concept may still care to consult for details the original paper of Perelman \cite{perelman} as well as the lecture notes of Topping \cite{top}. The introduction sections of both \cite{cao1} and \cite{li} also provide some motivation for using Perelman's functional in the context of the evolution of the eigenvalues of the operator $-\Delta+\dfrac{R}{2}$.

Our approach to the problem is rather by a direct assault. We  shall exploit in this paper some specific properties of tensors, especially of the metric tensor $g$, as well as some relations between them. Once with these formulae at hand we shall think of the Ricci flow, the Ricci-Bourguignon flow, and the Yamabe flow as ``one flow". Using in addition some properties of the Witten-Laplacian $L$  we shall directly derive our formula for the evolution of the eigenvalues of $-\mathbb{L}$.

\subsection{On differentiability of the eigenvalues}
All our computations in the sequel would certainly be illegitimate if the eigenvalues of the operators in question were not differentiable in time. Clearly, a $C^1$-differentiability of the eigenvalues will be sufficient for the purposes of the present work. Luckily, the latter is guaranteed by a theorem of A. Kriegl and P. Michor \cite{KriMich}. Precisely speaking, one has the following\begin{theorem*}[A. Kriegl \& P. Michor]
Let $t \mapsto A(t)$ for $t\in\mathbb{R}$ be a curve of unbounded self-adgoint operators  in a Hilbert space with common domain of definition and with compact resolvent.  
\newline\\
(A)\,\, If $A$ is $C^{\infty}$, then the eigenvalues of $A(t)$ may be parametrised twice differentiably in $t$.
\newline\\
(B)\,\, If $A(t)$ is $C^{1,\alpha}$ for some $\alpha >0$ in $t\in\mathbb{R}$, then the eigenvalues of $A(t)$ may be parametrised in a $C^1$ way in $t$.
\end{theorem*}
In fact, Kriegl and Michor prove more in their paper, but for the sake of brevity we shall briefly comment only on their results which are immediately relevant to our purposes. 

 A function $f$ is called $C^{k,\alpha}$ if it is $k$ times differentiable and for the $k$-th derivative the expression $\dfrac{f^{(k)}(t)-f^{(k)}(s)}{|t-s|^{\alpha}}$ is locally bounded in $t\neq s$. Now, as it is explicitly written in \cite{KriMich}, we have the following immediate consequences. Given a compact manifold $M$ consider the smooth curve $t\mapsto g_t$ of smooth Riemannian metrics on $M$. One then gets the corresponding smooth curve of $t\mapsto \Delta(g_t)$ of Laplace-Beltrami operators on 
 $L^2(M)$. Evidently, part $(A)$ of the theorem above guarantees that the eigenvalues can be arranged twice differentiably. Furthermore, one can consider a bounded region with smooth boundary $\Omega$ in $\mathbb{R}^n$ and the  $C^{1,\alpha}$-curve of Schr\"{o}dinger operators $S(t)=-\Delta + V(t)$ with varying potential function and Dirichlet boundary conditions. Then, by virtue of $(B)$, the eigenvalues can be arranged $C^1$.

\subsection{On general smooth variations of a metric $g(t)$.} The following comprises the main ingredient of our proof. For each $t$ we shall denote by $\h_{t}$ the symmetric $(0,2)$-tensor defined by $\h_{ij}(t)=\dfrac{d}{ds}\big|_{s=t}g_{ij}(s)$. Writing for its trace $h_{t}=\langle \h_{t},g_{t}\rangle$, one can easily check  that $\dfrac{d}{ds}\big|_{s=t}\dm_{s}=\dfrac{1}{2}h_{t}\dm_{t}$. This latter formula is essential as the measure itself will also vary in time. The evolution of the scalar curvature under a general smooth variation of the metric $g(t)$ is given by
\begin{eqnarray}\label{BBF}
\dfrac{\partial}{\partial t}R=-\langle Ric,\h\rangle-\Delta h+\mathrm{div}(\mathrm{div} \h).
\end{eqnarray}
This formula will be of utmost importance and a version of it can be found in  Besse \cite{besse}.

Now, for the purposes of this paper, one could think of $\h$ just as the flow \eqref{RBY}. Thus, we shall write
\begin{eqnarray}\label{BBS}
\h=-2aRic+2\bigg(\rho R-\frac{\varphi}{n}r\bigg)g\ \hbox{in} \ M\times[0,T).
\end{eqnarray}
By dint of this latter relation we can easily derive
\begin{eqnarray}\label{TR}
h=2\varphi(R-r) \ \hbox{in} \ M\times[0,T).
\end{eqnarray}
It is now evident that \eqref{BBF} can be rewritten in our context as
\begin{eqnarray}\label{AAW}
\dfrac{\partial}{\partial t}R=2a|Ric|^{2}+2\frac{R}{n}\bigg(\varphi r-\rho nR\bigg)+\psi\Delta R \ \hbox{in} \ M\times[0,T),
\end{eqnarray}
where $\psi=a-2(n-1)\rho$. 
Crucially important for us will also be the following formula.
\begin{equation}\label{impfor}
\frac{d}{dt}\Big|_{t=t_0}g(\nabla u,\nabla u)=2g(\nabla u',\nabla u)-\mathcal{H}(\nabla u,\nabla u).
\end{equation}
Here $u'$ means $\dfrac{\partial}{\partial t}u(x,t)$ (see {\bf Notation} at the end of this section).

\subsection{Some properties of the Witten-Laplacian.}
Of fundamental importance will also be some properties of the Witten Laplacian
 $L(\cdot)=\Delta-g(\nabla\eta,\nabla \cdot).$ 
It is not difficult to observe that it is self-adjoint in the Hilbert space $L^2(M,\dm)$. Notice also that for  all $f_1,f_2\in C_{0}^{\infty}(M)$ this operator  satisfies the following obvious property  
\begin{equation}\label{Lfg} 
L(f_1 f_2)=f_1Lf_2+f_2Lf_1+2g(\nabla f_1,\nabla f_2).
\end{equation}
Two important implications of this latter identity are as follows. Firstly, by virtue of this property it is almost immediate to check that if $-\mathbb{L}(u)=\lambda u$ then $$\mathbb{L}(u^{2})=-2\lambda u^{2}+2|\nabla u|^{2}-cRu^{2}.$$
We can rewrite this latter formula in the following more convenient form
\begin{eqnarray}\label{AAV}
-\lambda u^{2}+|\nabla u|^{2}-cRu^{2}=\frac{1}{2}L(u^{2}).
\end{eqnarray}
Secondly, it is well known that the formula of integration by parts for the operator $L$ is 
\begin{eqnarray} \label{int_pp}
\int_{M}\ell Lf\dm=-\int_{M}g(\nabla\ell,\nabla f)\dm,
\end{eqnarray}
for all $\,f,\ell\in C^{\infty}(M).$ Now, it is easy to perceive that \eqref{Lfg} and \eqref{int_pp} imply
\begin{equation}\label{cvx}
\int_{M}L(f^2)\dm=0,
\end{equation}
for any $f\in C^{\infty}(M)$.

\subsection{ Notation.} 
In order to keep our notation as simple as possible we shall often, but not always, omit writing the time variable (subscript) $t$, for  it will be tacitly assumed in this paper that all the operators, their eigenvalues and eigenfunctions as well as the metric and the scalar curvature evolve in time under some geometric flow. Thus, we shall rather write henceforth $\mathbb{L}$, $L$, $\lambda$, $u$, $R$, $g$, $\mathit{etc}$. In the same vein, we reserve the symbol $ ' $   to denote the derivatives in time and we shall frequently write $\lambda'$ instead of $\dfrac{d\lambda}{dt}(t)$, $R'$ instead of $\dfrac{\partial R}{\partial t}$, $\mathit{etc}$. The symbol $'$ will be used throughout, with a slight ambiguity, for denoting both ordinary and partial derivatives in time. However, it should be clear from the context which kind of derivative it would stand for in a given line of computation.


\section{Evolution of the eigenvalues}\label{evol}

In this section we shall study the evolution of the eigenvalues of the operator $-\mathbb{L}$ under the flow \eqref{RBY}. We begin with the proof of a core proposition which holds true for any variation $\h$ of the metric $g(t)$.
\begin{proposition}\label{P1}
Let $M^n$ be a closed Riemannian manifold with varying metric $g(t)$,\,$t\in[0,T)$. Assume that there is a $C^1$-family of normalised functions $u(x,t)$ such that $$-\mathbb{L}u(x,t)=\lambda (t)u(x,t).$$ Then,
\begin{eqnarray}\label{FVL}
\lambda'=\int_{M}\left[\frac{h}{4}L(u^{2})-\h(\nabla u,\nabla u)-cR'u^{2}\right]\dm.
\end{eqnarray} 
\end{proposition}
Remarkably, we have discovered two different proofs of this statement. We shall present in this section the one we think is the more elegant. Notwithstanding, we believe that the alternative proof, albeit longer, bears  its value and brings its contribution to this paper. For this reason, we have decided to convey it in an appendix.  
\begin{proof}
Integrating formula \eqref{AAV} and using the fact that $\int_M L(u^{2}) \dm=0$ we obtain
\begin{equation*}
\lambda (t) =\int_M|\nabla u|^{2}\dm-c\int_MRu^{2}\dm.
\end{equation*}
Differentiating this latter formula we get
\begin{eqnarray*}
\dfrac{d}{dt}\lambda(t)&=&\int_M\dfrac{d}{dt}|\nabla u|^{2}\dm+\dfrac{1}{2}\int_Mh|\nabla u|^{2}\dm-c\int_MR'u^{2}\dm \\\\ &-& c\int_MR(u^{2})'\dm-\dfrac{c}{2}\int_MhRu^{2}\dm.
\end{eqnarray*}
We next rewrite identity  \eqref{AAV} as 
$$\dfrac{h}{2}|\nabla u|^{2}-\dfrac{c}{2}hRu^{2}=\frac{h}{4}L(u^{2})+\dfrac{h}{2}\lambda u^{2}$$
and use formula  \eqref{impfor} to obtain
\begin{equation*}
\dfrac{d}{dt}\lambda(t)=\int_{M}\left[\frac{h}{4}L(u^{2})-\h(\nabla u,\nabla u)-cR'u^{2}\right]\dm + K,
\end{equation*}
where
\begin{equation*}
K=2\int_Mg(\nabla u', \nabla u)\dm -c\int_MR(u^2)'\dm+\dfrac{\lambda}{2}\int_Mhu^2\dm.
\end{equation*}
It remains to show that $K=0$. By means of the  integration by parts formula we can rewrite $K$ as
\begin{equation*}
K=-2\int_Mu'Lu\,\dm -c\int_MR(u^2)'\dm+\dfrac{\lambda}{2}\int_Mhu^2\dm.
\end{equation*}
As $-\mathbb{L}u=\lambda u$ implies $Lu= -\lambda u - cRu$ we immediately compute that
\begin{equation*}
K=2\lambda\int_Muu'\dm +\dfrac{\lambda}{2}\int_Mhu^2\dm=\lambda\dfrac{d}{dt}\int_Mu^2\dm\equiv0.
\end{equation*}
\end{proof}
Without much effort we have obtained the desired evolution formula in quite a compact  form. At the present moment, however, we are neither aware  of a feasible way to exploit it in order to proof monotonicity results nor we can clearly see that it indeed generalises the results of other authors as mentioned earlier. For this reasons we shall have to elaborate our formula a bit. 
By virtue of relations  \eqref{AAV} and \eqref{TR} we can write
\begin{eqnarray*}
\frac{h}{4}L(u^{2})=\varphi\frac{(R-r)}{2}L(u^{2})
=R\varphi\bigg(-\lambda u^{2}+|\nabla u|^{2}-cRu^{2}\bigg)-\frac{r\varphi}{2}L(u^{2}).
\end{eqnarray*}
Using relations \eqref{AAV} and  \eqref{BBS} we compute
\begin{eqnarray*}
\h(\nabla u,\nabla u)=-2aRic(\nabla u,\nabla u)+2\rho R|\nabla u|^{2}-2\frac{r\varphi}{n}\bigg(\frac{1}{2}L(u^{2})+\lambda u^{2}+cRu^{2}\bigg).
\end{eqnarray*}
It is immediate by \eqref{AAW}  that
\begin{eqnarray*}
cR'u^{2}=c\Big(2a|Ric|^{2}+2\frac{R}{n}(\varphi r-\rho nR)+\psi\Delta R\Big)u^{2}.
\end{eqnarray*}
It now readily follows that the integrand in relation \eqref{FVL} can be written as
\begin{eqnarray*}
I&=&\frac{2r\varphi}{n}\lambda u^{2}-\lambda\varphi Ru^{2}+(\varphi-2\rho)R|\nabla u|^{2}+2aRic(\nabla u,\nabla u)\\
&&-\frac{r\varphi(n-2)}{2n}L(u^{2})-c\Big(2a|Ric|^{2}+(\varphi -2\rho)R^{2}+\psi\Delta R\Big)u^{2}. 
\end{eqnarray*}
We have thus arrived at the principal theorem of this section and one of the main results in this paper.
\begin{theorem}\label{T1}
Let $g(t)$,\,$t\in[0,T)$, be a solution to the flow \eqref{RBY} on a closed manifold $M^n$. Assume that there is a $C^1$-family of normalised functions $u(x,t)$ such that $-\mathbb{L}u(x,t)=\lambda (t)u(x,t)$. Then, for $\psi=a-2(n-1)\rho$, the eigenvalues $\lambda(t)$ satisfy
 \begin{eqnarray}\label{FVRBY}
\nonumber\dfrac{d}{dt}\lambda(t)&=&\frac{2\lambda r}{n}\varphi-\lambda\varphi\int_{M}u^{2}R\dm
+(\varphi-2\rho)\int_{M}R|\nabla u|^{2}\dm\\\\\nonumber&&+2a\int_{M}Ric(\nabla u,\nabla u)\dm-c\int_{M}\bigg[R^{2}(\varphi-2\rho)+2a|Ric|^{2}+\psi\Delta R\bigg]u^{2}\dm.
\end{eqnarray}
\end{theorem}
Evidently, the latter theorem is merely a reincarnation of the evolution formula \eqref{FVL}. Albeit not as beautiful, it is  in a sense more enlightening. Indeed, taking $a=1, r=0, \rho=0, c=0$ we perceive that formula \eqref{FVRBY} reduces to
$$\dfrac{d}{dt}\lambda(t)=\lambda\int_{M}u^{2}R\,\dm
-\int_{M}R|\nabla u|^{2}\dm+2\int_{M}Ric(\nabla u,\nabla u)\dm.$$
Remarkably, this is exactly the evolution  formula for the eigenvalues of the Laplacian under unnormalised Ricci flow (see Proposition 2.1 in \cite{Fab}).  In fact, Theorem \ref{T1} nicely encompasses many other particular cases. We invite the reader at this juncture to check that formula \eqref{FVRBY} reduces to the corresponding evolution formulas in the following works \cite{cao1, cao2, cao3, fang3, Zeng}. 

We shall close this section by swiftly proving  two corollaries of Theorem \ref{T1}. Again, they can be thought of as extensions of two of Di Cerbo's corollaries (see \cite{Fab}). It must be clear  at this point that, in much the same spirit, one can easily derive other corollaries reminiscent to those already published in the literature.
\begin{corollary}\label{cor1}
Let $\lambda (t)$ be an eigenvalue of $-\mathbb{L}$ associated to the eigenfunction $u(x,t)$ under the flow \eqref{RBY}. If the scalar curvature $R$ is constant along this flow for each $t\in[0,T)$, then
\begin{eqnarray}\label{A2}
\dfrac{d}{dt}\lambda(t)=-\frac{2a\lambda R}{n}+2a\int_{M}Ric(\nabla u,\nabla u)\dm-2ac\int_{M}u^{2}|Ric|^{2}\dm
\end{eqnarray}
in $M\times[0,T)$.
\end{corollary}
\begin{proof}
If $R$ is constant then $r=R$, and relation \eqref{FVRBY} becomes
\begin{eqnarray*}
\lambda'&=&\frac{2\lambda R}{n}\varphi-\lambda\varphi R+R(\varphi-2\rho)\int_{M}|\nabla u|^{2}\dm+2a\int_{M}Ric(\nabla u,\nabla u)\dm\\&-&cR^{2}(\varphi-2\rho)-2ac\int_{M}u^2|Ric|^{2}\dm.
\end{eqnarray*}
Using \eqref{AAV} one computes
\begin{eqnarray*}
\int_{M}|\nabla u|^{2}\dm=\lambda +cR,
\end{eqnarray*}
which, along with the fact that $\varphi=-a+n\rho$, implies relation \eqref{A2}.
\end{proof}

Notice that a use of this  corollary will be made later in Section \ref{appl}. 

\begin{corollary}\label{cor2}
Let $\lambda (t)$ be an eigenvalue of $-\mathbb{L}$ associated to the eigenfunction $u(x,t)$ under the flow \eqref{RBY} on a compact surface $(M^{2},g)$. Then, the following variation formula is true in $M\times[0,T)$
\begin{eqnarray}\label{A14}
\nonumber\dfrac{d}{dt}\lambda(t)&=&\lambda r\varphi-\lambda\varphi\int_{M}u^{2}R\dm+c\varphi\int_{M}u^{2}\Delta R\dm.
\end{eqnarray}
\end{corollary}
\begin{proof}
Clearly, $n=2$ implies  $\varphi-2\rho=-a$, $\psi=-\varphi$, $Ric=\dfrac{R}{2}g$ and $|Ric|^{2}=\dfrac{R^{2}}{2}$. Now, plugging these latter relations in equation \eqref{FVRBY} yields the desired formula.
\end{proof}


\section{Monotonicity of the eigenvalues}\label{mono}

 A slight modification of the flow \eqref{RBY} and imposing some assumptions on the curvature will yield that the eigenvalues of the Witten - Laplacian $L$ are non-decreasing.  To be able to state the first monotonicity result we shall  define the following quantity
$$\mathfrak{T}_{t}=\dfrac{2R_{t}}{n}\left(\dfrac{\varphi r_{t}}{n}-\rho R_{t}\right)g_{t}+(\psi-1)\nabla^{2}R_{t}.$$ 
Here, and henceforth, the symbol $\nabla^2$ is always to mean the Hessian. Also, by $\langle \mathfrak{T},g\rangle\geq 0 \ \hbox{in} \ M\times[0,T)$ we shall mean that the inequality $\langle \mathfrak{T}_t,g_t \rangle\geq 0$ holds true for all $t\in[0,T)$. Notice that this condition holds true trivially in the case of unnormalised Ricci flow  as $\mathfrak{T}=0$.

\begin{proposition}\label{A16}
Let $R_{t}$ be the evolution of the scalar curvature under the flow given by \eqref{RBY}. Let $K\in\mathbb{R}$ such that $R\geq K \ \hbox{in} \ M\times\{0\}$. If $\langle \mathfrak{T},g\rangle\geq 0 \ \hbox{in} \ M\times[0,T)$ and $a\geq 0$, then
$R\geq K \ \hbox{in} \ M\times[0,T)$.
\end{proposition}

\begin{proof}
By hypothesis
\begin{eqnarray*}
\frac{2R}{n}\bigg(\varphi r-\rho nR\bigg)+\psi\Delta R\geq\Delta R \ \hbox{in} \ M\times[0,T).
\end{eqnarray*}
On the other hand, relation \eqref{AAW} implies
\begin{eqnarray*}
R'=2a|Ric|^{2}+2\frac{R}{n}\bigg(\varphi r-\rho nR\bigg)+\psi\Delta R \ \hbox{in} \ M\times[0,T).
\end{eqnarray*}
Since $a\geq 0$, it follows that $R'\geq\Delta R \ \hbox{in} \ M\times[0,T)$. The desired result now follows directly from the maximum principle (see Proposition 2.9 in \cite{chow}).
\end{proof}

In order to prove our first monotonicity result we shall confine ourselves to the flow $\mathcal{F}$ which is defined by formula \eqref{RBY} with $a\geq 0$ and $\rho\leq\dfrac{a}{2(n-1)}$. Albeit a tad restricted, this flow still encompasses the Ricci, the Ricci - Bourguignon and the Yamabe flows which is sufficient for the generalising purposes of this paper.  

 \begin{theorem}\label{T2}
Let $\big(M^{n},g(t)\big)$, $t\in[0,T)$, be a solution of the flow $\mathcal{F}$ on a closed manifold of dimension $n>2$ and let $\lambda(t)$ be an eigenvalue of  $-L$. Suppose that there exists a non-negative constant $\mathcal{A}$ satisfying the inequalities
\begin{eqnarray}\label{A11}
 aRic+\frac{(\varphi-2\rho)}{2}Rg\geq-\mathcal{A}g, \ \hbox{in} \ M\times[0,T),
 \end{eqnarray}
and
\begin{eqnarray}\label{A12}
R\geq 2\left(\frac{-\mathcal{A}}{\varphi}+\frac{r}{n}\right), \ \hbox{in} \ M\times\{0\}.
\end{eqnarray}
If $\langle \mathfrak{T}_t,g\rangle\geq 0 \ \hbox{in} \ [0,T)$, then $\lambda(t)$ is monotone increasing along this flow.
\end{theorem}
\begin{proof}
Applying Theorem \ref{T1} for $c=0$, and using relations \eqref{A11}, \eqref{AAV}, and \eqref{cvx} we compute 
\begin{eqnarray*}
\lambda'&=&\frac{2\lambda r\varphi}{n}-\lambda\varphi\int_{M}u^{2}R\,\dm+(\varphi-2\rho)\int_{M}R|\nabla u|^{2}\dm+2a\int_{M}Ric(\nabla u,\nabla u)\dm\\
&\stackrel{\eqref{A11}}{\geq}&\frac{2\lambda r\varphi}{n}-\lambda\varphi\int_{M}u^{2}R\dm-2\mathcal{A}\int_{M}|\nabla u|^{2}\dm\\
&\stackrel{\eqref{AAV}\&{\eqref{cvx}}}{=}&\frac{2\lambda r\varphi}{n}-\lambda\varphi\int_{M}u^{2}R\dm-\lambda\int_{M}2\mathcal{A}u^{2}\dm\\
&=&\lambda\int_{M}\bigg(-\varphi R-2\mathcal{A}+\frac{2r\varphi}{n}\bigg)u^{2}\dm.
\end{eqnarray*}
Observe now that $n>2$, $a\geq0$ and $\rho<\dfrac{a}{2(n-1)}$ imply $\varphi<0$. Considering the inequalities $\langle \mathfrak{T},g\rangle\geq 0 \ \hbox{in} \ [0,T)$ and $R\geq 2\left(\frac{-\mathcal{A}}{\varphi}+\frac{r}{n}\right) \ \hbox{in} \ M\times\{0\}$ we easily infer by Proposition \ref{A16} that
\begin{eqnarray*}
-\varphi R-2\mathcal{A}+\frac{2r\varphi}{n}\geq 0 \ \hbox{in} \ M\times[0,T),
\end{eqnarray*}
and the proof is complete.
\end{proof}
It will certainly be worth generalising the discussion above in the case of the more general operator $\mathbb{L}$ evolving with the flow $\mathcal{F}$, or even with \eqref{RBY}. However, we are only able at the present moment to do so by imposing further restrictions. In particular, an assumption on the drifting function $\eta$ will be necessary. Write $\mathcal{G}$ for the flow \eqref{RBY} with the following constraints

$$r=0,\, -\dfrac{1}{4}<c<0, \, a\geq0 \,\,\,\mathrm{and}\,\,\, \rho<\frac{a}{2(n-1)}.$$

In this case, the following theorem holds true.
 

\begin{theorem}\label{monn}
Let $g(t)$, $t\in[0,T)$, be a solution to the flow $\mathcal{G}$ on a closed
manifold $M^{n}$, $n>2$. Assume that the drifting function $\eta$ satisfies $\Delta\eta\leq\dfrac{(\varphi-2\rho)}{\psi}R$ for all $t\in[0,T)$. If there exists a constant $\mathcal{A}\geq 0$ such that for all $t\in[0,T)$
$$\big[(\varphi-2\rho)-c\psi\big]Rg+2aRic\geq-\mathcal{A}g $$
and 
$$R\geq\frac{\mathcal{A}}{2c\psi-\varphi},$$ 
then the eigenvalues of $-\mathbb{L}$ are nondecreasing along this flow.
\end{theorem}
For the sake of brevity of the exposition we shall present the proof of this theorem in the appendix.


\section{An upper bound for the first variation of $\lambda(t)$}\label{appl}

In this section we shall discuss one immediate application of the preceding two sections. We shall  first derive a generalisation of the Reilly formula for the operator $\mathbb{L}$. Then, by dint of this latter formula, an upper bound for the first variation of the eigenvalue $\lambda(t)$ will be obtained. By the end of this section the following notation shall also be adopted. Given a manifold $M$ and a drifting function $\eta\in C^{\infty}(M)$ we define the symmetric $(0,2)$-tensor $$\mathcal{T}=Ric_{\eta}-\dfrac{cR}{2}g,$$ where $Ric_{\eta}=Ric+\nabla^{2}\eta$ is the Bakry-Emery-Ricci tensor.

\begin{lemma}\label{lem}
Given a closed manifold $M$ and a function $f\in C^{\infty}(M)$, the following formula holds true
\begin{eqnarray*}
 \int_{M}(\mathbb{L}f)^{2}\dm&=&\int_{M}\Big[\mathcal{T}(\nabla f,\nabla f)+|\nabla^{2}f|^{2}\Big]\dm +c\int_{M}R\bigg(\mathbb{L}(f^{2})-\frac{3}{2}|\nabla f|^{2}\bigg)\dm.
\end{eqnarray*}
\end{lemma}
\vspace{0.5cm}
\begin{remark}
For obvious reasons, this  formula will be henceforth referred to as the Reilly formula. 
\end{remark}
\begin{proof}
We shall first derive a Bochner-type formula for the operator $\mathbb{L}$. Our starting point will be the obvious identity
$$\dfrac{1}{2}\mathbb{L}|\nabla f|^{2}=\Delta\dfrac{1}{2}|\nabla f|^{2}-g\big(\nabla\eta,\nabla\dfrac{1}{2}|\nabla f|^{2}\big)+\dfrac{c}{2}R|\nabla f|^{2}.$$
Firstly, applying the usual Bochner formula for the standard Laplacian we can write
$$\dfrac{1}{2}\mathbb{L}|\nabla f|^{2}=|\nabla^2 f|^{2}+Ric(\nabla f,\nabla f)+g\big(\nabla f, \nabla (\Delta f)\big)-g\big(\nabla\eta,\nabla\dfrac{1}{2}|\nabla f|^{2}\big)+\dfrac{c}{2}R|\nabla f|^{2}.$$
Then, the replacement of  $\Delta f$ with $\mathbb{L}f+g(\nabla\eta,\nabla f)-cRf$ in the third summand gives
\begin{eqnarray*}
\frac{1}{2}\mathbb{L}|\nabla f|^{2}&=&g\big(\nabla f,\nabla(\mathbb{L}f)\big)+Ric(\nabla f,\nabla f)+|\nabla^{2}f|^{2}-\frac{1}{2}g\big(\nabla\eta,\nabla|\nabla f|^{2}\big)\\
&&-\frac{cR}{2}|\nabla f|^{2}+g\big(\nabla f,\nabla g(\nabla\eta,\nabla f)\big)-cfg(\nabla R,\nabla f).
\end{eqnarray*}
Finally, by means of the well-known relation $\nabla|\nabla f|^{2}=2\nabla_{\nabla f}\nabla f$, we obtain our Bochner-type formula to be
\begin{equation}\label{BOC}
\frac{1}{2}\mathbb{L}|\nabla f|^{2}=\mathcal{T}(\nabla f,\nabla f)+g\big(\nabla f,\nabla(\mathbb{L}f)\big)+|\nabla^{2}f|^{2}-\frac{c}{2}g\big(\nabla R,\nabla f^{2}\big).
\end{equation}
\newline\\
To complete the proof of the lemma we just need to integrate formula \eqref{BOC}. Before doing so, however, we shall have to make two crucial observations. Firstly,  by the definition of $\mathbb{L}$ and  relation \eqref{cvx} we can write
\begin{equation}\label{BBA}
\frac{1}{2}\int_{M}\mathbb{L}|\nabla f|^{2}\dm=\dfrac{c}{2}\int_{M}R|\nabla f|^{2}\dm.
\end{equation}
Secondly, we need to integrate the term $g\big(\nabla f,\nabla(\mathbb{L}f)\big)$. For this purpose we shall write
$$\big(\mathbb{L}f\big)^2=\big(\mathbb{L}f\big)(Lf+cRf)=\big(\mathbb{L}f\big)(Lf)+cRf\big(\mathbb{L}f\big).$$
Clearly, $\mathbb{L}f$ can be thought of as a function itself and thus we can apply our integration by parts formula \eqref{int_pp} to the term $\big(\mathbb{L}f\big)(Lf)$. Whence,
\begin{equation}\label{BBB}\int_{M}g\big(\nabla f,\nabla(\mathbb{L}f)\big)\dm=-\int_{M}(\mathbb{L}f)^{2}\dm+c\int_{M}Rf(\mathbb{L}f)\dm.
\end{equation}
Now, integrating the Bochner formula \eqref{BOC}, and using relations \eqref{BBA} and \eqref{BBB}   we can write
\begin{eqnarray}
\nonumber\int_{M}(\mathbb{L}f)^{2}\dm&=&\int_{M}\mathcal{T}(\nabla f,\nabla f)\dm+\int_{M}|\nabla^{2}f|^{2}\dm\\
\nonumber&&+\frac{c}{2}\int_{M}\bigg(2Rf(\mathbb{L}f)-R|\nabla f|^{2}-g(\nabla R,\nabla f^{2})\bigg)\dm.
\end{eqnarray}
 Writing $Rf\mathbb{L}f = RfLf+cR^2f^2$, and applying the integration by parts formula completes the proof.
  \end{proof}

Now, by exploiting Lemma \ref{lem} we establish an upper-bound estimate for  
$\dfrac{d}{dt}\lambda (t)$.
\begin{theorem}\label{T1A}
Let $(M,g)$ be a closed manifold and $\lambda (t)$ be the evolution of an eigenvalue of $-\mathbb{L}$ under the flow \eqref{RBY} with $a>0$. If the scalar curvature $R$ is constant along the flow, then
\begin{eqnarray*}
\nonumber\dfrac{d}{dt}\lambda(t)&\leq&2a\dfrac{(n-1)}{n}\lambda^{2}+\dfrac{2aR}{n}\big(2nc-2c-1\big)\lambda+2a\dfrac{(n-1)}{n}c^2R^2 \\\\
\nonumber&&-2ac\int_{M}u^2|Ric|^{2}\dm-\frac{2a}{n}\int_{M}g(\nabla\eta,\nabla u)^{2}\dm-2a\int_{M}\nabla^{2}\eta(\nabla u,\nabla u)\dm\\\\
&&+\frac{2a(\lambda+cR)}{n}\int_{M}g(\nabla\eta,\nabla u^{2})\dm
\end{eqnarray*}
in $M\times[0,T)$.
\end{theorem}

\begin{proof}
Since the scalar curvature is constant, we can obviously write by Lemma \ref{lem} 
\begin{eqnarray*}
\int_{M}|\nabla^{2}u|^{2}\dm=\int_{M}(\mathbb{L}u)^{2}\dm-\int_{M}\mathcal{T}(\nabla u,\nabla u)\dm
-cR\int_{M}\bigg(\mathbb{L}u^{2}-\frac{3}{2}|\nabla u|^{2}\bigg)\dm.
\end{eqnarray*}
Recalling that $-\mathbb{L}u=\lambda u$ we immediately have that
 \begin{equation*}
 \int_{M}(\mathbb{L}u)^{2}\dm=\lambda^{2}. 
 \end{equation*}
 We also effortlessly compute the following two quantities
\begin{eqnarray*}
{\bf A)}\,\,\,\int_{M}\mathcal{T}(\nabla u,\nabla u)\dm&=&\int_{M}Ric(\nabla u,\nabla u)\dm+\int_{M}\nabla^{2}\eta\big(\nabla u,\nabla u\big)\dm\\
\nonumber&&-\frac{cR}{2}\int_{M}|\nabla u|^{2}\dm\\
&\stackrel{\eqref{A2}}{=}&\frac{\lambda'}{2a}+\frac{R\lambda}{n}+c\int_{M}u^2|Ric|^{2}\dm\\
\nonumber&&+\int_{M}\nabla^{2}\eta\big(\nabla u,\nabla u\big)\dm-\frac{cR}{2}\int_{M}|\nabla u|^{2}\dm.
\end{eqnarray*}

\begin{eqnarray*}
{\bf B)}\,\,\,cR\int_{M}\bigg(\mathbb{L}u^{2}-\frac{3}{2}|\nabla u|^{2}\bigg)\dm=-2\lambda cR+\frac{cR}{2}\int_{M}|\nabla u|^{2}\dm-c^{2}R^{2}.
\end{eqnarray*}
We have thus obtained the following formula for the integral of the square of the Hessian
\begin{eqnarray}\label{inthess}
\int_{M}|\nabla^{2}u|^{2}\dm&=&-\dfrac{\lambda'}{2a}+\lambda^2+\dfrac{\lambda R}{n} \Big(2nc-1\Big)+c^2R^2\\
&-&c\nonumber\int_{M}u^2|Ric|^{2}\dm-\int_{M}\nabla^{2}\eta\big(\nabla u,\nabla u\big)\dm.
\end{eqnarray}
It is well-known that the Cauchy-Schwarz inequality implies $n|\nabla^{2}u|^{2}\geq(\Delta u)^{2}$. We also have by definition that $\Delta u=\mathbb{L}u+g(\nabla\eta,\nabla u)-cRu$. Then, the following inequality emerges
\begin{eqnarray*}
\int_{M}|\nabla^{2}u|^{2}\dm\geq\frac{1}{n}\int_{M}(\Delta u)^{2}\dm=\frac{1}{n}\int_{M}\bigg(\mathbb{L}u+g(\nabla\eta,\nabla u)-cRu\bigg)^{2}\dm, 
\end{eqnarray*}
which can be expanded further as
\begin{eqnarray*}
\nonumber\int_{M}|\nabla^{2}u|^{2}\dm&\geq&\frac{1}{n}\int_{M}\bigg((\mathbb{L}u)^{2}+g(\nabla\eta,\nabla u)^{2}+c^{2}R^{2}u^{2}+2(\mathbb{L}u)g(\nabla\eta,\nabla u)\bigg)\dm\\
&&-\frac{1}{n}\int_{M}\bigg(2cRu\mathbb{L}u+2cRug(\nabla\eta,\nabla u)\bigg)\dm.
\end{eqnarray*}
It is now a matter of a direct calculation to perceive the truth of 
\begin{eqnarray*}
\nonumber\int_{M}|\nabla^{2}u|^{2}\dm\geq\dfrac{1}{n}(\lambda+cR)^{2}-\dfrac{1}{n}(\lambda+cR)\int_{M}g(\nabla\eta,\nabla u^{2})+\dfrac{1}{n}\int_{M}g(\nabla\eta,\nabla u)^{2}\dm.
\end{eqnarray*}
Now, the estimate for $\lambda'$ stated in the proposition follows immediately from this latter inequality and formula \eqref{inthess}.
\end{proof}

\section{Conjectural answer to Question \ref{qu1}}\label{homotopy}

We have already seen above that Theorems \ref{T1} and  \ref{T2} generalise most of the results already known in the literature. More importantly, the non-decreasing  property of the eigenvalues of given geometric operator evolving under some geometric flow have been always proven under certain hypotheses. In our opinion, this observations constitute enough empirical evidence to formulate the following conjectural answer of Question \ref{qu1}:
\newline\\
 { \it The monotonicity of the eigenvalues of the operators $\Delta+cR$, $L$ and $\mathbb{L}$ is inherited from the monotonicity of the eigenvalues of the Laplace - Beltrami operator $\Delta$.}  
\newline\\
Firslty, it is not difficult to observe that the eigenvalues of both $-\Delta$ and $-L$ obey exactly the same evolution formula under the flow \eqref{RBY}. Indeed, in both cases we have the following evolution formula
 
$$\dfrac{d}{dt} \lambda(t)=\dfrac{2\lambda r}{n}\varphi-\lambda\varphi\int_{M}u^{2}R\dm+(\varphi-2\rho)\int_{M}R|\nabla u|^{2}\dm+2a\int_{M}Ric(\nabla u,\nabla u)\dm.$$

Secondly, we can think of the operator $-L$ as a Schr\"{o}dinger operator with potential function
$V(x)=\langle\nabla\eta,\nabla\cdot\rangle$. 

With these remarks in mind we would like to propose the following conjectures.


\begin{conj}
Let $S$ be a Schr\"{o}dinger operator evolving with a geometric flow $\mathcal{F}$. If the eigenvalues of $-\Delta$ are non-decreasing under the flow $\mathcal{F}$, then so are the eigenvalues of $S$. 
\end{conj}
Let us recall further, that {\it unitarily equivalent} operators have by definition the same spectra. Therefore, it is also natural to expect that the following is true.
\begin{conj}
Suppose that the eigenvalues of a Schr\"{o}dinger operator $S$ are non-decreasing under a geometric flow $\mathcal{F}$. Then, the eigenvalues of any unitarily equivalent operator to $S$ are also non-decreasing under the flow $\mathcal{F}$. 
\end{conj}




\section{Few Concluding Remarks}\label{concl}
To finish, we should like to pose two questions which we think are both interesting and important, and we are not aware of their answers. 
We have assumed throughout this article that equation \eqref{RBY} admitted unique short-time solutions. However, we are aware of the fact that this is not the case for arbitrary values of the parameters $a$ and $\rho$. For instance, it is well-known that the equation 
\begin{equation*}
\dfrac{\partial g}{\partial t}=-Ric+\dfrac{R}{2}g
\end{equation*}  
turns out to behave badly from a PDE viewpoint (see \cite{top}). Therefore, the following question naturally arises.
\begin{qu}
For which values of the constants $a$ and $\rho$ the flow
\begin{equation*}
\frac{\partial g_{t}}{\partial t}=-2aRic_{t}+2\left(\rho R_{t}-\frac{\varphi}{n}r_{t}\right)g_{t},
\end{equation*}
admits unique short-time solutions for arbitrary initial data?
\end{qu}
The second natural question is regarding the monotonicity.
\begin{qu}
Given a geometric operator evolving with a geometric flow, are there any obstructions for the monotonicity of its eigenvalues? If so, are these obstructions a manifestation of the flow, or they are rather a consequence of the topology of the manifold? 
\end{qu}


\section{Appendix}\label{appendix}
\subsection{An alternative proof of the evolution formula}
As promised earlier in Section \ref{evol}, we shall now give an alternative proof of the evolution formula for the eigenvalues of the operator $-\mathbb{L}$ evolving with the flow \eqref{RBY}. This proof conceptually differs from the one we have already presented, for it makes use of the notion of the {\it $\eta$-divergence} and some of its properties. Given  a symmetric $(0,r)$-tensor  $S$ on $M$  we define the {\it $\eta$-divergence} to be the $(0,r-1)$-tensor 
\begin{eqnarray*}
\mathrm{div}_{\eta}S=\mathrm{div} S-\mathrm{d}\eta\circ S,
\end{eqnarray*}
where $\mathrm{div}S$ is the usual divergence of S and $\eta$ is the drifting function. The identity $\mathrm{div}_{\eta}(\nabla f)=L(f)$ is effortlessly checked. The $\eta$-divergence enjoys the property $\mathrm{div}(e^{-\eta}X)=e^{-\eta}\mathrm{div}_{\eta}X$ for any smooth function $f$ and a vector field $X$ on the manifold $M$. It is by its virtue that we have a natural extension of the divergence theorem for $\mathrm{div}_{\eta}$. Indeed, given a vector field $X$ on $M$ and the weighted measure $\mathrm{d}\mu=e^{-\eta}\mathrm{d}\partial M$ on the boundary $\partial M$, the divergence theorem takes the form 
$$\int_M\mathrm{div}_{\eta}X \dm=\int_{\partial M}g(X,\nu)\mathrm{d}\mu,$$
where $\nu$ is the unit outward normal vector. We shall also need in what follows the  property
$$\mathrm{div}_{\eta}(fX)=f\mathrm{div}_{\eta}X+g(\nabla f, X).$$

Taking any function $f\in C_{0}^{\infty}(M)$ and working in coordinates one can straightforwardly verify the identity 
\begin{equation}\label{derL}
\frac{d}{dt}\bigg(\mathbb{L}f\bigg)=\mathbb{L}'f+\mathbb{L}f'.
 \end{equation}
 Moreover, the following formula holds true (see \cite{GMR}) 
  \begin{eqnarray}\label{A7}
\mathbb{L}'f=\bigg\langle\frac{1}{2}\mathrm{d}h-\mathrm{div}_{\eta}\h,\mathrm{d}f\bigg\rangle-\langle \h,\nabla^2f\rangle+cR'f.
\end{eqnarray}

Now, with the aforementioned preparatory remarks in mind we can give the second prove of the evolution formula \eqref{FVL}. 
\begin{proof}
To begin with, observe that by differentiating in $t$  the obvious identity $$-u(t)\mathbb{L}u(t)=\lambda(t)u^2(t),$$ we easily obtain
\begin{eqnarray*}
u'\mathbb{L}u-u\mathbb{L}u'=\lambda'u^{2}+u\mathbb{L}'u,
\end{eqnarray*}
which reduces to
\begin{eqnarray*}
u'Lu-uLu'=
\lambda' u^{2}+u\mathbb{L}'u.
\end{eqnarray*}
By dint of formula \eqref{int_pp} we readily perceive that the left hand side vanishes after the integration and therefore the following simple formula holds true 
\begin{eqnarray}\label{A9}
\lambda'=-\int_{M}u\mathbb{L}'u\,\dm.
\end{eqnarray}
It is now evident that we only need to compute the integral on the right hand side. One can calculate this integral in two different ways. The obvious one is to use the variation formula for the operator $\mathbb{L}$. Namely, plugging \eqref{A7} in \eqref{A9} and using the properties of $\mathrm{div}_{\eta}$ as well as the divergence theorem one will eventually arrive at the desired formula. Leaving this computation for the reader to check we shall give another way of computing the integral in \eqref{A9}. We observe first that the definition of $\mathbb{L}$  and the integration by part formula yield 
\begin{equation}\label{B}\int_{M}u\mathbb{L}u\,\dm=-\int_{M}\langle \mathrm{d}u,\mathrm{d}u\rangle \dm+c\int_{M}Ru^{2}\dm.\end{equation}
Evidently, we only need to differentiate the latter formula and simplify to formula \eqref{FVL}. Using the identity \eqref{derL} we can rewrite the left hand side integral of \eqref{B} as
\begin{equation*}
 \dfrac{d}{dt}\left(\int_{M}u\mathbb{L}u\,\dm\right)=\int_{M}u'(\mathbb{L}u)\dm+\int_{M}u\mathbb{L}'u\,\dm+\int_{M}u\mathbb{L}u'\,\dm+\frac{1}{2}\int_{M}hu\mathbb{L}u\,\dm.
\end{equation*}
Differentiation of the right hand side of \eqref{B} gives
\begin{eqnarray*}
\nonumber\dfrac{d}{dt}\left(\int_{M}u\mathbb{L}u\,\dm\right)&=&-2\int_{M}\langle \mathrm{d}u,\mathrm{d}u'\rangle \dm+\int_{M}\mathcal{H}(\nabla u,\nabla u)\dm-\frac{1}{2}\int_{M}h|\nabla u|^{2}\dm\\
&&+c\int_{M}R'u^{2}\dm+c\int_{M}R(u^{2})'\dm+\frac{c}{2}\int_{M}hRu^{2}\dm.
\end{eqnarray*}
It is also easily seen that the integration by parts formula implies
\begin{eqnarray*}
\nonumber\int_{M}u'\mathbb{L}u\,\dm+\int_{M}u\mathbb{L}u'\dm&=&\int_{M}u'Lu\,\dm+\int_{M}uLu'\,\dm+c\int_{M}R(u^{2})'\dm\\
&=&-2\int_{M}\langle \mathrm{d}u,\mathrm{d}u'\rangle \dm+c\int_{M}R(u^{2})'\dm.
\end{eqnarray*}
Thus, the latter three formulae  justify the validity of
\begin{eqnarray}\label{E}
\nonumber\int_{M}u\mathbb{L'}u\,\dm&=&\int_{M}\mathcal{H}(\nabla u,\nabla u)\dm-\frac{1}{2}\int_{M}h\langle \mathrm{d}u,\mathrm{d}u\rangle \dm\\
&&+c\int_{M}R'u^{2}\dm+\frac{c}{2}\int_{M}Rhu^{2}\dm-\frac{1}{2}\int_{M}hu\mathbb{L}u\,\dm.
\end{eqnarray}

We shall have to make a little detour at this point. Using the properties of $\mathrm{div}_{\eta}$ we readily compute
\begin{eqnarray*}
\nonumber \mathrm{div}_{\eta}(hu\nabla u)&=&huLu+g(\nabla hu,\nabla u)\\
\nonumber&=&huLu+h\langle  \mathrm{d}u, \mathrm{d}u\rangle+u\langle \mathrm{d}h, \mathrm{d}u\rangle\\
\nonumber&=&hu(\mathbb{L}u-cRu)+h\langle  \mathrm{d}u, \mathrm{d}u\rangle+u\langle \mathrm{d}h, \mathrm{d}u\rangle\\
&=&hu\mathbb{L}u-cRhu^{2}+h\langle  \mathrm{d}u, \mathrm{d}u\rangle+u\langle \mathrm{d}h, \mathrm{d}u\rangle.
\end{eqnarray*}
Now, as we are on a closed manifold, $\int_M\mathrm{div}_{\eta}(hu\nabla u) \dm$ vanishes and we obtain
$$
\int_{M}h\langle  \mathrm{d}u, \mathrm{d}u\rangle \dm=-\int_{M}u\langle \mathrm{d}h, \mathrm{d}u\rangle \dm-\int_{M}hu\mathbb{L}u\,\dm+c\int_{M}Rhu^{2}\dm.
$$
It follows that relation \eqref{E} can be expressed as
$$
\int_{M}u\mathbb{L}'u\,\dm=\int_{M}\mathcal{H}(\nabla u,\nabla u)\dm+\frac{1}{2}\int_{M}u\langle \mathrm{d}h, \mathrm{d}u\rangle \dm+c\int_{M}R'u^{2}\dm.
$$
Recall now that the Witten-Laplacian  can be defined as $\mathrm{div}_{\eta}(\nabla f)=L(f)$. This implies the identity
\begin{equation*}
\frac{u}{2}\langle \mathrm{d}h, \mathrm{d}u\rangle=\frac{1}{4}\mathrm{div}_{\eta}(h\nabla u^{2})-\frac{1}{4}hL(u^{2}),
\end{equation*}
which, by dint of the divergence theorem, can be rewritten in the following integral form
$$\int_M\frac{u}{2}\langle \mathrm{d}h, \mathrm{d}u\rangle \dm=-\int_M\frac{1}{4}hL(u^{2}) \dm.$$

We can then write
$$
\int_{M}u\mathbb{L}'u\,\dm=\int_{M}\left(-\frac{h}{4}L(u^{2})+\mathcal{H}(\nabla u,\nabla u)+cR'u^{2}\right)\dm,
$$
which completes the proof.

\end{proof}


\subsection{Proof of Theorem \ref{monn}}
In order to prove this theorem we shall need to rewrite formula \eqref{FVRBY} from Theorem \ref{T1}. Without much effort one can justify that
\newline\\
\begin{eqnarray*}
\dfrac{d\lambda}{dt}&=&\lambda \int_{M}(2c\psi-\varphi)Ru^{2}\dm+c\big[2c\psi-(\varphi-2\rho)\big]\int_{M}R^{2}u^{2}\dm
\nonumber-2ac\int_{M}|Ric|^{2}u^{2}\dm\\\\&&-\big[c\psi-(\varphi-2\rho)\big]\int_{M}R|\nabla u|^{2}\dm+2a\int_{M}Ric(\nabla u,\nabla u)\dm\\\\
&&+c\psi\int_{M}Ru^{2}\Delta\eta\, \dm-c\psi\int_{M}R|\nabla u-u\nabla\eta|^{2}\dm.
\end{eqnarray*}
\newline\\
Notice that it was assumed here that $r=0$. If $r\neq 0$, then additional term $\dfrac{2\lambda r\varphi}{n}$ will appear. Notice that this formula can also be very handy in order to check that Theorem \ref{T1} indeed generalises all the evolution formulas appearing in the works cited in the bibleography.

Now, we are in a position to prove that $\dfrac{d \lambda}{dt}\geq0$. The proof consists of the following six steps.
\newline\\
{\bf STEP1.}
The hypotheses $n>2, \ -\frac{1}{4}<c<0, \ \rho<\dfrac{a}{2(n-1)}$ imply the crucial inequality  $2c\psi-\varphi>0.$ To perceive the truth of the latter we first observe that the assumption $-\frac{1}{4}<c<0$ readily implies the following inequality
\newline
$$\frac{n}{2c+1}>\frac{2}{2c+1}.$$
\newline
We then have the following sequence of inequalities
\newline
$$\left[2-\left(\frac{4c+1}{2c+1}\right)\right]n>2-\frac{4c}{2c+1};$$ \\
$$2n-2>\left(\frac{4c+1}{2c+1}\right)n-\frac{4c}{2c+1};$$\\
$$2(n-1)>\frac{1}{2c+1}[4c(n-1)+n];$$\\
$$\rho<\frac{a}{2(n-1)}<\frac{a(2c+1)}{4c(n-1)+n};$$\\
$$\big[4c(n-1)+n\big]\rho<a(2c+1).$$\\ 
Now, $2c\psi-\varphi>0$ follows immediately as $\varphi=-a+n\rho$ and $\psi=a-2(n-1)\rho$.
\newline\\
{\bf STEP2.} Clearly, $\mathcal{A}\geq 0$ and $R\geq\displaystyle\frac{\mathcal{A}}{2c\psi-\varphi}$ imply
\begin{equation}\label{V}
\int_{M}[(2c\psi-\varphi)R-\mathcal{A}]u^{2}dm\geq 0.
\end{equation}
\newline\\
{\bf STEP3.} Notice that $R\geq\displaystyle\frac{\mathcal{A}}{2c\psi-\varphi}$ and STEP1 imply
$R\geq 0$. The assumption $\rho<\dfrac{a}{2(n-1)}$ implies $\psi>0$. Hence,
\begin{equation}\label{W}
-c\psi\int_{M}R|\nabla u-u\nabla\eta|^{2}dm\geq 0.
\end{equation}
\newline\\
{\bf STEP4.}  The following estimate holds true for $a\geq 0$ and $c<0$
\begin{equation}\label{X}
-2ac\int_{M}|Ric|^{2}u^{2}\dm\geq 0.
\end{equation}
\newline\\
{\bf STEP5.} The hypotheses $\big[(\varphi-2\rho)-c\psi\big]Rg+2aRic\geq-\mathcal{A}g$ and
$\psi>0$ along with the equations \eqref{AAV} and \eqref{cvx} yield
\begin{equation}\label{Y}
-[c\psi-(\varphi-2\rho)]\int_{M}R|\nabla u|^{2}\dm+2a\int_{M}Ric(\nabla u,\nabla u)\dm\geq \lambda\int_{M}-\mathcal{A}u^{2}\dm.
\end{equation}
\newline\\
{\bf STEP6.}
In the final step we shall use the hypotheses on the drifting function $\eta$. Recall that $c<0$, $R\geq 0$ and $\psi>0$. These latter along with $\Delta\eta\leq\dfrac{(\varphi-2\rho)R}{\psi}$ imply the following sequence of inequalities
\begin{eqnarray*}
&&\psi\Delta\eta\leq(\varphi-2\rho)R\\
&&\psi\Delta\eta\leq[(\varphi-2\rho)-2c\psi]R\\
&&c\psi\Delta\eta\geq c[(\varphi-2\rho)-2c\psi]R\\
&&c\psi R\Delta\eta\geq c[(\varphi-2\rho)-2c\psi]R^{2}\\
&&c\psi R\Delta\eta\geq -c[2c\psi-(\varphi-2\rho)]R^{2}\\
&&c\psi Ru^{2}\Delta\eta\geq -c[2c\psi-(\varphi-2\rho)]R^{2}u^{2},
\end{eqnarray*}
which readily imply
\begin{equation}\label{AA}
c\psi\int_{M}Ru^{2}\Delta\eta \,\dm+c[2c\psi-(\varphi-2\rho)]\int_{M}R^{2}u^{2}\dm\geq 0.
\end{equation}
The desired monotonicity of $\lambda$ is now evident.




\end{document}